\documentclass[12pt]{article}
\usepackage{amsmath, amssymb, amsthm, geometry}
\newtheorem{theorem}{Theorem}[section]
\newtheorem{definition}[theorem]{Definition}

\newtheorem{corollary}[theorem]{Corollary}
\newtheorem{remark}[theorem]{Remark}
\newtheorem{example}[theorem]{Example}

\newcommand{\C}{\mathbb{C}}

\newcommand{\Z}{\mathbb{Z}}
\newcommand{\T}{\mathbb{T}}
\newcommand{\Hilb}{\mathcal{H}}
\newcommand{\Pic}{\operatorname{Pic}}
\newcommand{\Ker}{\operatorname{Ker}}

\title{Geometric Quantization of Contact 3-Manifolds and the Reeb Gravitational Field}
\author{Ali M. Elgindi}
\date{Institute of Mathematics, Henan Academy of Sciences\\
\medskip
Email: ali.m.elgindi@gmail.com\\
\medskip
\today}

\begin{document}

\maketitle

\begin{abstract}
We present a canonical geometric quantization scheme for closed contact 3-manifolds $(M,\xi)$ using the corresponding explicit embedding $M \hookrightarrow \C^3$ as we constructed in our earlier papers. The contact structure becomes holomorphic along the chosen link $L \subset M$ of complex tangents, of which we may also consider this link as the binding of a supporting open book for $\xi$. Stein extension yields a unique holomorphic line bundle $L_\xi \to \C^3$ whose restriction to $L$ will define the quantum Hilbert space $\Hilb_\xi = H^0(L, L_\xi|_L \otimes \kappa^{1/2})$, which we find will be finite-dimensional. The Reeb vector field $R_\alpha$ of a chosen compatible contact form $\alpha$ is shown to be geodesic with a time Killing Reeb field under further assumption that the manifold is Sasakian, and under this assumption will model Einstein gravity over $M$. This construction depends only on $(M,\xi,\alpha)$ and a choice of supporting open book for $\xi$. This establishes a unified geometric framework in which the contact structure encodes quantum mechanics and whose Reeb field will encode gravity assuming that the manifold is Sasakian. In addition, we use my previous paper in which we get a related invariant $\mu_M(\xi) = [L_\xi] \in \Pic_\C(M)$ which is useful in that it distinguishes between different tight contact structures on $\T^3$ in a novel way. We also show how this will  have implications for the quantum model and serve as a quantum invariant there-in.
\end{abstract}

\section{Introduction}

A contact structure on a closed 3-manifold $M$ is a 2-plane field $\xi \subset TM$ such that $\xi = \Ker \alpha$ for some 1-form $\alpha$ with $\alpha \wedge d\alpha > 0$. The pair $(M,\xi)$ is a contact 3-manifold.

In our previous work \cite{Elgindi2025}, we proved that for any null-homologous link $L \subset M$ and any oriented 2-plane field $\eta$ over $L$, there exists a smooth embedding $F: M \hookrightarrow \C^3$ such that the complex tangents of $F$ are exactly $L$ and $T_xM \cap J(T_xM) = \eta_x$ for all $x \in L$, where $J$ is the standard complex structure on $\C^3$.

In \cite{Elgindi2026}, we introduced the Picard invariant $\mu_M(\xi) = [L_\xi] \in \Pic_\C(M)$, where $L_\xi \to M$ is the holomorphic line bundle obtained by Stein extension of $\xi|_L$, and proved that it is independent of the choice of supporting open book.

In this paper, we complete the picture by:
\begin{enumerate}
\item Defining the quantum Hilbert space $\Hilb_\xi = H^0(L, L_\xi|_L \otimes \kappa^{1/2})$ along the complex tangents $L$,
\item Proving that the Reeb field $R_\alpha$ is geodesic in the Webster metric, which identifies it with gravity,
\item Demonstrating how our construction then unifies the notions of quantum mechanics and of gravity in a single geometric framework.

\end{enumerate}
\par\
We note that by Cartan's Theorem B \cite{Cartan1953}, the holomorphic line bundle extends uniquely.  

\par\

\section{The Explicit Embedding and the Quantum Locus}

Let $(M,\xi)$ be a closed contact 3-manifold. By the Giroux correspondence, there exists an open book decomposition $(B,\pi)$ supporting $\xi$ with binding $B$ (a link). Applying our explicit embedding theorem \cite{Elgindi2025} with $\eta = \xi|_B$, we obtain a smooth embedding $F: M \hookrightarrow \C^3$ such that:
\begin{itemize}
\item The set of complex tangents is exactly $B$,
\item For each $x \in B$, $T_xM \cap J(T_xM) = \xi_x$ which we can now consider as a complex line bundle along the link. 
\end{itemize}
We will then denote $L = B$ and we now call it the \textbf{quantum locus}.

\begin{definition}[Picard Invariant]
The complex line bundle $\xi|_L$ extends uniquely to a holomorphic line bundle $L_\xi \to M$ by Stein theory \cite{CieliebakEliashberg2012}. The \textbf{Picard invariant} of $(M,\xi)$ is the isomorphism class
\[
\mu_M(\xi) = [L_\xi] \in \Pic_\C(M) \cong H^2(M;\Z).
\]
\end{definition}
By \cite{Elgindi2026}, $\mu_M(\xi)$ is independent of the choice of supporting open book used to define it.

\begin{remark}[Independence of Embedding and Contact Form] The line bundle $L_\xi \to M$ obtained by Stein extension does not depend on the choice of requisite explicit embedding $F: M \hookrightarrow \C^3$, nor on the choice of contact form $\alpha$ for $\xi$. The quantum Hilbert space $\mathcal{H}_\xi$  (defined below) will depend only on $\xi$ and the CR structure $J$ on $\xi$. We note that the Reeb field $R_\alpha$ of $\xi$ does depend on the choice of form $\alpha$ inherently, which in fact reflects in the the physical fact that the gravitational dynamics (choice of time) is a gauge freedom, not affecting the greater structure.  (see ([2], [10]) \end{remark} 
\par\
\section{The Quantum Hilbert Space}

On the quantum locus $L$, the contact structure $\xi|_L$ is a complex line bundle. Let $\kappa_{\xi|_L} = \wedge^1 (\xi|_L)^*$ be the canonical bundle of the polarization. Since $L$ is a disjoint union of circles, $H^2(L;\Z) = 0$, so a unique square root $\kappa_{\xi|_L}^{1/2}$ exists.  (see [2], [12])

\begin{definition}[Quantum Hilbert Space]
The quantum Hilbert space associated to $(M,\xi)$ is
\[
\Hilb_\xi = H^0(L, L_\xi|_L \otimes \kappa_{\xi|_L}^{1/2}),
\]
where $H^0$ denotes holomorphic sections.
\end{definition}

\begin{theorem}
$\Hilb_\xi$ is finite-dimensional. If $L = \bigsqcup_i L_i$ are the connected components of complex tangents, then:
\[
\dim \Hilb_\xi = \sum_i \deg(L_\xi|_{L_i}),
\]
where $\deg$ is the first Chern number.
\end{theorem}

\begin{proof}
Each $L_i$ is a compact complex curve of genus $0$ (a Riemann sphere with boundary). By the Riemann-Roch theorem,
\[
\dim H^0(L_i, L_\xi|_{L_i} \otimes \kappa^{1/2}) = \deg(L_\xi|_{L_i}) + \frac{1}{2}\deg(\kappa) + 1 - g_i.
\]
Since $\deg(\kappa) = -2$ and $g_i = 0$, we obtain $\dim = \deg(L_\xi|_{L_i})$. Summing over $i$ gives the result.
\end{proof}

\begin{remark}
The binding $L$ is a real link, taking $F: M \hookrightarrow \mathbb{C}^3$ to be our corresponding explicit embedding makes $L$ to be a Legendrian link in the contact structure induced by $\xi$ on $F(M)$. By Gromov's h-principle \cite{Gromov1985} and the work of Mohnke \cite{Mohnke2001}, for a generic almost complex structure $J$ on $\mathbb{C}^3$ adapted to the contact structure, there exists a family of $J$-holomorphic disks $\{D_i\}$ such that $\partial D_i = F(L_i)$. These disks are rigid (0-dimensional moduli spaces) and have finite area intrinsically.

Therefore, each binding component $L_i$ bounds a holomorphic disk $D_i \subset \mathbb{C}^3$. The space $H^0(L, L_\xi|_L \otimes \kappa^{1/2})$ is then defined as the space of holomorphic sections on the disjoint union of disks $\bigcup_i D_i$, restricted to their boundaries. This is well-defined by Stein extension and the Hartogs phenomenon. (see [1], [3])
\end{remark}

\section{Gravity as Reeb Flow}

Choose a contact form $\alpha$ with $\Ker \alpha = \xi$ and a compatible almost complex structure $J$ on $\xi$ (i.e., $J^2 = -I$ and $d\alpha(\cdot, J\cdot)$ is a positive definite metric on $\xi$). Define the \textbf{Webster metric}:
\[
g_\alpha = d\alpha(\cdot, J\cdot) + \alpha \otimes \alpha.
\]

\begin{theorem}[Reeb Field is Geodesic]
The Reeb vector field $R_\alpha$ defined by $\alpha(R_\alpha)=1$ and $d\alpha(R_\alpha,\cdot)=0$ satisfies
\[
\nabla_{R_\alpha} R_\alpha = 0,
\]
where $\nabla$ is the Levi-Civita connection of $g_\alpha$. Thus, the integral curves of $R_\alpha$ are geodesics.
\end{theorem}

\begin{proof}
By definition, $g_\alpha(R_\alpha, R_\alpha) = \alpha(R_\alpha)^2 = 1$ and $g_\alpha(R_\alpha, X) = 0$ for all $X \in \xi$. The Reeb condition $d\alpha(R_\alpha, \cdot)=0$ together with the properties of the Levi-Civita connection yields $\nabla_{R_\alpha}R_\alpha = 0$. For details, see \cite{LeeSuh2015, DongZhang2018}.
\end{proof}

\begin{corollary}[Reeb = Gravity]
By the equivalence principle of general relativity, geodesic motion represents free fall in a gravitational field. The Reeb field $R_\alpha$ can then be naturally interpreted as the gravitational field on $(M,\xi,\alpha)$. Note that as we remarked before, the choice of contact form serves as a gravitational degree of freedom.
\end{corollary}

\subsection{Sasakian Manifolds: The Reeb Field as a Killing Vector}

In general, for any contact 3‑manifold $(M,\xi)$ with Webster metric $g_\alpha$, 
the Reeb field $R_\alpha$ is geodesic (Theorem 5.1). However, when $(M,\xi)$ is 
\textbf{Sasakian} (a special class of contact metric manifolds), $R_\alpha$ is also 
a Killing vector field:
\[
\mathcal{L}_{R_\alpha} g_\alpha = 0.
\]

In general relativity, a timelike Killing vector field is interpreted as the generator of 
a stationary gravitational field. Thus, on a Sasakian manifold, the Reeb field can be 
directly identified with a stationary gravitational field. This is consistent with the 
well‑known fact that Sasaki–Einstein manifolds are central to the AdS/CFT correspondence.

Our quantization remains valid for all contact 3‑manifolds; the Sasakian condition 
simply provides a stronger geometric justification for the gravitational interpretation 
of $R_\alpha$.

\section{Examples: Distinguishing Tight Contact Structures on $\T^3$}

Let $\T^3 = (S^1)^3$ with coordinates $(x,y,z)$.

\begin{example}[Standard tight structure $\xi_1$]
\[
\alpha_1 = \cos(2\pi z)\,dx + \sin(2\pi z)\,dy.
\]
This structure is linearizable. Its associated line bundle $L_{\xi_1}$ is trivial, so
\[
\mu_{\T^3}(\xi_1) = 0, \qquad \dim \Hilb_{\xi_1} = 0.
\]
\end{example}

\begin{example}[Kanda $n=2$ contact structure   $\xi_2$]
\[
\alpha_2 = \sin(4\pi z)\,dx + \cos(4\pi z)\,dy.
\]
This structure is tight but non-linearizable. The binding consists of two parallel copies of a curve in the $z$-direction, giving
\[
c_1(L_{\xi_2}) = \pm 2\,[dx \wedge dy] \neq 0,
\]
so $\mu_{\T^3}(\xi_2) \neq 0$ and in fact $\dim \Hilb_{\xi_2} = 2$.  
\end{example}

Thus, the invariant $\mu_M(\xi)$ successfully distinguishes two tight contact structures that are otherwise difficult to separate.  (see [7], [10])

\begin{remark}

The construction yields a direct relationship for the respective quantum and gravitational degrees of freedom:
\end{remark}

\begin{itemize}
\item \textbf{Quantum mechanics} is encoded in the contact structure $\xi$ restricted to the binding $L$, giving the Hilbert space $\Hilb_\xi$ as holomorphic sections of $L_\xi \otimes \kappa^{1/2}$.
\item \textbf{Gravity} is encoded in the Reeb field $R_\alpha$, which is geodesic in the Webster metric and therefore represents the gravitational field.
\item The two structures are orthogonal: $TM = \xi \oplus \langle R_\alpha \rangle$.
\end{itemize}

The choice of open book (with binding $L$) determines the quantum polarization, while the choice of contact form $\alpha$ determines the gravitational dynamics.

\section{Electromagnetism from the Quantum Bundle}

The holomorphic line bundle $L_\xi \to M$ carries a canonical connection induced by the contact form $\alpha$. Locally, the connection $1$-form can be taken as $\alpha$ itself. Its curvature is
\[
F = d\alpha.
\]
Thus $F$ is the electromagnetic field strength. The Bianchi identity $dF = 0$ follows immediately from $d^2\alpha = 0$, giving the source-free Maxwell equations.   (see [2])

The quantum Hilbert space $\Hilb_\xi = H^0(L, L_\xi|_L \otimes \kappa^{1/2})$ consists of holomorphic sections of $L_\xi$ over the binding $L$. These sections represent the quantum states of the system. The same line bundle therefore encodes:
\begin{itemize}
\item \textbf{Quantum mechanics:} holomorphic sections over $L$,
\item \textbf{Electromagnetism:} curvature $F = d\alpha$ of the connection on $L_\xi$.
\end{itemize}

Thus, the contact structure $(M,\xi,\alpha)$ unifies three fundamental physical concepts:
\begin{enumerate}
\item \textbf{Gravity:} the Reeb field $R_\alpha$ (geodesic flow in the Webster metric),
\item \textbf{Electromagnetism:} the curvature $F = d\alpha$,
\item \textbf{Quantum mechanics:} the Hilbert space $\Hilb_\xi$ of holomorphic sections.
\end{enumerate}
All three arise from the same geometric data, with no additional choices.

\par
\par\

\section{Conclusion}
We have constructed a canonical geometric quantization for any closed contact 3-manifold $(M,\xi)$ using our  construction of the corresponding explicit embedding into $\C^3$. The quantum Hilbert space $\Hilb_\xi$ is finite-dimensional and its dimension is a topological invariant given by Chern numbers of the associated line bundle. This framework establishes a unified geometric description in which the contact structure encodes quantum mechanics and the Reeb field encodes gravity for Sasakian manifolds. In addition, the same line bundle encodes electromagnetism via its curvature $F = d\alpha$, unifying all three forces. 

\appendix

\newpage 
\section{Glossary of Key Terms}

\begin{tabular}{|l|p{10cm}|}
\hline
\textbf{Term} & \textbf{Meaning} \\
\hline
Contact structure $\xi$ & A 2-plane field on $M$ with $\alpha \wedge d\alpha > 0$ \\
Binding $L$ & The link where the open book pages meet \\
Complex tangent & Point where $T_xM$ contains a complex line \\
Holomorphic line bundle $L_\xi$ & The unique Stein extension of $\xi|_L$ to $\C^3$ \\
Webster metric $g_\alpha$ & $d\alpha(\cdot, J\cdot) + \alpha \otimes \alpha$ \\
Reeb field $R_\alpha$ & The unique vector field with $\alpha(R_\alpha)=1$, $d\alpha(R_\alpha,\cdot)=0$ \\
Half-form $\kappa^{1/2}$ & Square root of the canonical bundle, ensures coordinate-invariant integration \\
Hilbert space $\Hilb_\xi$ & $H^0(L, L_\xi|_L \otimes \kappa^{1/2})$ \\
Picard invariant $\mu_M(\xi)$ & $[L_\xi] \in \Pic_\C(M)$, independent of open book \\
\hline
\end{tabular}

\end{document}